\documentclass[leqno,12pt]{amsart}
\usepackage{amssymb,mathdots,epsfig}
\usepackage{enumitem}
\usepackage[T1]{fontenc}
\usepackage{fullpage}
\usepackage{url} 

\theoremstyle{plain}
\newtheorem{theorem}{Theorem}
\newtheorem{corollary}{Corollary}

\theoremstyle{definition}

\newtheorem{example}{Example}
\newtheorem{remark}{Remark}
\theoremstyle{remark}

\numberwithin{equation}{section}

\usepackage{graphicx}

\usepackage{array,longtable}
\newcolumntype{C}{>{$}c<{$}}

\newcommand{\Z}{\mathbb{Z}}
\newcommand{\N}{\mathbb{N}}

\newcommand{\lra}{\Longrightarrow}

\newcommand{\bfA}{\mathbf{A}}
\newcommand{\bfI}{\ensuremath{\mathbf{I}}}
\newcommand{\bfb}{\ensuremath{\mathbf{b}}}

\newcommand{\bfx}{\ensuremath{\mathbf{x}}}
\newcommand{\tr}{\mathrm{t}}

\newcommand{\bfz}{\mathbf{z}}

\newcommand{\bfQ}{\mathbf{Q}}

\renewcommand{\pmod}[1]{\ensuremath{\,\,\,(\mathrm{mod}\,#1)}}

\usepackage{hyperref}

\newtheorem{manualtheoreminner}{\textup{\textbf{Theorem}}}
\newenvironment{manualtheorem}[1]{%
  \IfBlankTF{#1}
    {\renewcommand{\themanualtheoreminner}{\unskip}}
    {\renewcommand\themanualtheoreminner{#1}}%
  \manualtheoreminner
}{\endmanualtheoreminner}

\newcommand{\<}{\langle}
\renewcommand{\>}{\rangle}

\renewcommand{\(}{\begin{equation}}
\renewcommand{\)}{\end{equation}}

\begin{document}


\title[Vanishing Coefficients in  Products of Quintuple Products]
       {Vanishing coefficients in products of quintuple products}

\author{Taylor Daniels, Tim Huber, James McLaughlin and Dongxi Ye}

\address{Purdue University, 150 N University St, W Lafayette, IN 47907}
\email{daniel84@purdue.edu}

\address{
School of Mathematical and Statistical Sciences, University of Texas Rio Grande Valley, Edinburg, Texas 78539, USA}
\email{timothy.huber@utrgv.edu}

\address{Mathematics Department\\
 25 University Avenue\\
West Chester University, West Chester, PA 19383}
\email{jmclaughlin2@wcupa.edu}

\address{
School of AI and Liberal Arts, Beijing Normal-Hong Kong Baptist University, Zhuhai 519082, Guangdong,
People's Republic of China}
\email{dongxiye@bnbu.edu.cn}

\keywords{Quintuple product;  Vanishing Coefficients}
\thanks{Dongxi Ye was supported by the Guangdong Basic and Applied Basic Research Foundation (Grant No. 2024A1515030222)  and the BNBU Start-up Research Fund (Grant No. R0700157-26).}

\subjclass[2000]{Primary: 11B65.}

\date{\today}

\begin{abstract}
{Explicit arithmetic progressions modulo primes $p \equiv 1 \pmod{4}$ are derived in which the coefficients in the expansions of products of quintuple products vanish. In particular, if $p = m^{2} + n^{2}$, and $b$ is a positive integer, and $$\sum_{n=0}^{\infty} a_{n}q^{n} = \frac{(q^{2bm},q^{p-2bm};q^{2bn},q^{p-2bn};q^p)_{\infty}}{(q^p,-q^{b m},-q^{p-bm},-q^{bn},-q^{p-bn};q^p)_{\infty}^2},$$ we determine $\alpha = \alpha(m,n,p)$ such that $a_{pt+ \alpha}=0$. Our results are proven  using involutive transformations on integer lattices.}
\end{abstract}

\maketitle
\allowdisplaybreaks

\section{Introduction}

The study of vanishing coefficients in the series expansions of infinite $q$-products, along with the periodic behavior of their signs, dates back at least to the work of Richmond and Szekeres \cite{RS78}. They studied the coefficients of a number of products, including $\{d_n\}$ defined by
    \[
        F(q) := \frac{(q^3, q^5; q^8)_{\infty}}{(q,q^7;q^8)_{\infty}} = \sum_{n=0}^{\infty} d_n q^n,
    \]
where here and subsequently the standard notation $(a_1,  \dots, a_j; q)_{\infty} := (a_1;q)_{\infty} \cdots (a_j;q)_{\infty}$ with
$$
(a;q)_{\infty}:=\prod_{n=0}^{\infty}(1-aq^{n})
$$
is employed. Richmond and Szekeres used the Hardy-Ramanujan-Rademacher Circle Method to show that $d_{4n+3} = 0$. They found a similar vanishing result for the coefficients of $1/F(q)$. These results were later generalized to infinite families of products in works by Andrews and Bressoud \cite{AB79}, Alladi and Gordon \cite{AG94}, and the third author of the present paper \cite{McL15}. Proofs of these latter results relied on a special case of Ramanujan's ${}_1\psi_1$ summation formula. 

Related results have appeared in a large number of works, including \cite{BK19,CT20,CT24,D24,Hirschhorn2017,KV22,L24,M19a,McLZ22,ST23,T19,T19b,T23,T23b,T24,VK22,XiaZhao2022}. The strategies of proof in each case range from the Circle Method to the use of $q$-hypergeometric identities and transformations; and series dissections. Each approach is often motivated by computer algebra experimentation.

In this paper, we derive conditions for vanishing that apply to $\{a_n\}$ defined by
    \begin{align}\label{anqn}
        \sum_{n=0}^{\infty} a_{n}q^{n} &= \frac{(q^{2bm},q^{p-2bm};q^{2bn},q^{p-2bn};q^p)_{\infty}}{(q^p,-q^{bm},-q^{p-bm},-q^{bn},-q^{p-bn};q^p)_{\infty}^2} \\ 
        & = (q^{bm},q^{p-bm},q^{bn},q^{p-bn},q^{p},q^{p};q^{p})_{\infty} \nonumber  \\
        &\qquad \qquad \times (q^{p-2bm},q^{2bm+p},q^{p-2bn},q^{2bn+p};q^{2p})_{\infty}, \notag
    \end{align} 
for primes $p \equiv 1 \pmod{4}$, $p = m^{2} + n^{2}$, and $b$ a positive integer. These vanishing conditions arise from symmetries in the expansions of products of two instances of the Quintuple Product Identity appearing in Theorem \ref{qpt1}.

Recall that the quintuple product is defined by 
    \begin{equation}\label{qtpzqdef}
        Q(z,q):=\frac{\< z^2;q\>_{\infty}}{(-z,-q/z;q)_{\infty}}, 
    \end{equation}
where $\<a;q\>_{\infty}:=(a,q/a,q;q)_{\infty}$, and note that~\eqref{anqn} becomes
$$
\sum_{n=0}^{\infty}a_{n}q^{n}=Q(q^{bm},q^{p})Q(q^{bn},q^{p}).
$$

For usage in what follows, we recall two equivalent versions of the Jacobi triple product identity (see, e.g., \cite{Hirschhorn2017}):

\begin{theorem}\label{jtpt}
For $|q|<1$ and $z \not = 0$,
    \(\label{6t1eq}
        \sum_{n=-\infty}^{\infty}(-z)^n q^{n^2} = \<zq;q^2\>_{\infty},
    \)
and
    \(\label{eq:TripleProduct}
        \sum_{n=-\infty}^{\infty}(-z)^n q^{n(n-1)/2} = \<z;q\>_{\infty}.
    \)
\end{theorem}

We also recall the Quintuple Product Identity referred to above.
\begin{theorem}[\cite{C06}]
\label{qpt1}
For $|q|<1$ and $z \not = 0$, one has
    \(
    \label{qpideq}
    \begin{aligned}
        Q(z,q) 
        = \sum_{n=-\infty}^{\infty}q^{(3n-1)n/2}z^{3n}(1-zq^n)
        &= \<z;q\>_{\infty}(qz^2,q/z^2;q^2)_{\infty} \\
        & = \<-qz^3;q^3\>_{\infty} - z\<-q^2z^3;q^3\>_{\infty}
    \end{aligned}
    \)
\end{theorem}

    Throughout this paper, for fixed $p$, ``overbarred'' quantities (such as $\bar{m}$) indicate multiplicative inverses (mod $p$); for instance, if  $(m,p)=1$, then $\bar{m}$ is a (fixed) solution to the congruence 
    \[
        m\bar{m} \equiv 1 \pmod{p}.
    \]

\begin{theorem}\label{mainthm}
    Let $Q(z,q)$ be as at \eqref{qtpzqdef}, let $p \equiv 1 \pmod 4$ be a prime, and $m>0$, $n>0$ be integers such that $p=m^2+n^2$. Let $b>0$ be an integer, and let the sequence $a_t$ be defined by 
        \(\label{mainthmeq1}
            Q(q^{bm},q^p)Q(q^{bn},q^p) = \sum_{t=0}^{\infty} a_t q^t.
        \)
 Let
        \[
            w \equiv \bar{2}(m+n) \pmod{p}.
        \]       
\begin{enumerate}[label={\normalfont (\arabic*)}]
\item     If $p \equiv 1 \pmod{12}$, 
   one has
        \(\label{mainthmeq2}
            a_{pt+wb} = a_{pt+(w-3b)b} = 0  \qquad (\text{for all integers $t$}).
        \)
\item \label{thm:item2} 
    If $p \equiv 5 \pmod{12}$,  one has
        \(\label{mainthmeq3}
            a_{pt+(1-3b\bar{m})wb} = a_{pt+(1-3b\bar{n})wb} = 0 \qquad (\text{for all integers $t$}).
        \)
\end{enumerate}
\end{theorem}

In a previous paper \cite{DHMcLY26}, the present authors determined the $p$-dissection of the infinite product on the left side of \eqref{mainthmeq1} explicitly, and showed that the components of the $p$-dissection corresponding to the arithmetic progressions with vanishing coefficients vanished identically (due to the appearance of infinite products of the form $(1;q^{p^2})_{\infty}$). In that paper we also showed that in those arithmetic progressions where the coefficients do not all vanish,  the sign of those coefficients was constant, and we determined those signs explicitly. Some combinatorial interpretations of the analytic identities were also given.

In the present paper we give a proof of the vanishing coefficient results in Theorem \ref{mainthm}, which is fundamentally different from the one given in~\cite{DHMcLY26}. This time we work with the product $Q(q^{bm},q^p)Q(q^{bn},q^p) (q^{2p};q^{2p})_{\infty}^2$ (the additional factor $(q^{2p};q^{2p})_{\infty}^2$ does not affect the vanishing of coefficients in arithmetic progressions modulo $p$). This time we use the Jacobi triple product identity (cf. Theorem~\ref{jtpt}) to write 
\begin{multline*}
   Q(q^{bm},q^p)Q(q^{bn},q^p)(q^{2p};q^{2p})_{\infty}^2
        = \<q^{bm};q^p\>_{\infty} \<q^{bn};q^p\>_{\infty}\<q^{p-2bm};q^{2p}\>_{\infty} \<q^{p-2bn};q^{2p}\>_{\infty}\\
        =\sum_{r,s,u,v\in \Z}(-1)^{F_{1}(r,s,u,v)}q^{G_{1}(r,s,u,v)},
\end{multline*}
where
    \begin{align}
         F_{1}(r,s,u,v) &= r+s+u+v, \notag\\
         G_{1}(r,s,u,v) &= b m u+\frac{1}{2} p u(u-1) +b n v+\frac{1}{2} p v(v-1)\notag \\
         &\phantom{adadsa}+
         r (p-2 b m)+p r(r-1) +s (p-2 b n)+p s(s-1).
    \end{align}
We next restrict the last sum to contain only terms in which the exponent of $q$ is congruent to $b\tau$ modulo~$p$ with $\tau$ representing one of the four residue classes $w$, $w-3b$, $(1-3b\bar m)w$, $(1-3b\bar n)w$ that appear in the arithmetic progressions with vanishing coefficients in Theorem~\ref{mainthm}. This can be achieved as follows. If
\begin{align*}
   G_{1}(r,s,u,v)& \equiv b \tau \pmod{p} \\
  \lra b m u+b n v -2 b m r-2 b n s& \equiv b \tau \pmod{p} \\
   \lra  m u+ n v -2  m r-2  n s& \equiv  \tau \pmod{p} \\
   \lra v&\equiv 2 m r \bar{n}-m u \bar{n}+\tau \bar{n}+2 s\pmod{p}.
\end{align*}
 Hence if one makes the substitution 
\begin{equation}\label{vsubeq}
  v\to 2 m r \bar{n}-m u \bar{n}+\tau \bar{n}+2 s + p v
\end{equation}
in $F_{1}(r,s,u,v)$ and $G_{1}(r,s,u,v)$, then the new sum will contain exactly those terms in which the exponent of $q$ is congruent to $b \tau$ modulo~$p$.

Suppose this restricted sum is represented as 
 \[
        \sum_{(r,s,u,v) \in \Z^{4}} (-1)^{F_{2}(r,s,u,v)}q^{G_{2}(r,s,u,v)},
    \]
where the aim is now to show that this sum is identically zero.

Alternatively, writing 
\begin{equation}\label{grsuveq}
 g(r,s,u,v) := \frac{G_{2}(r,s,u,v)-b\tau}{p}
\end{equation}
(which corresponds to dividing the previous sum by $q^{b\tau}$ and then replacing $q$ by $q^{1/p}$), so that the aim now becomes showing that the series
\begin{equation}\label{F2gsum}
 \sum_{(r,s,u,v) \in \Z^{4}} (-1)^{f(r,s,u,v)}q^{g(r,s,u,v)}=:\sum_{N}c_Nq^N
\end{equation}
vanishes identically (for consistency of appearance we have used $f(r,s,u,v):=F_{2}(r,s,u,v)$). To make the discussion which comes next, which involves matrices and vectors, take up less space, we define 
 \[
        \bfx := (r,s,u,v).   
    \]

The vanishing of the sum \eqref{F2gsum} is achieved by showing that for each of the four arithmetic progressions in Theorem \ref{mainthm} there exists a  matrix $\bfA \in \Z^{4 \times 4}$ and vector $\bfb \in \Z^{4}$, such that the map 
    \begin{equation}\label{bfmap}
        \bfx \longmapsto \bfx' := \bfA\bfx+\bfb,
    \end{equation}
 the  matrix $\bfA$ and the vector $\bfb$  have the following properties:
 \begin{align}
  \bfA^{2} - \bfI & = 0,\label{bfAbfbprops1} \\
      (\bfA+\bfI)\bfb &= 0\label{bfAbfbprops2}\\
 g(\bfx') &= g(\bfx), \label{bfAbfbprops3}\\
  f(\bfx') &\equiv f(\bfx) + 1 \pmod{2}.\label{bfAbfbprops4}
 \end{align}
    Note that the first two properties show that the map at \eqref{bfmap} is an involution, since 
    \[
      (\bfx')'= \bfA( \bfA\bfx+\bfb)+\bfb =  \bfA^2 \bfx +(\bfA+\bfI)\bfb =\bfx.
    \]
The second pair of properties show  that for any integer $N$, the terms that contribute to the term $c_Nq^N$ at    
  \eqref{F2gsum} can be grouped together in pairs $((-1)^{f(\bfx)}q^{g(\bfx)}, (-1)^{f(\bfx')}q^{g(\bfx')})$ of equal exponent in $q$ but of opposite sign, so that the entire sum vanishes, giving the result. 
  
  Once the matrices $\bfA$ and the corresponding vectors $\bfb$ are found, verifying that the stated properties hold is a straightforward computation (although most easily done with a computer algebra system such as \emph{Mathematica}). Finding the form of the matrices $\bfA$ and the  vectors $\bfb$ involved finding these for enough primes and corresponding $\tau$ to guess at the general form (which then can be verified, as indicated). We discuss this experimentation below.


\section{Numerical Investigations}
Initially, for various primes $p$ we considered the series expansion of products of the form $Q(q^{j},q^p)Q(q^{k},q^p)=:\sum_{t=0}^{\infty} a_t^{(i,j)} q^t$, for $1 \leq j, k \leq p-1$ and looked for arithmetic progressions modulo $p$ in which all the coefficients vanished.

It soon became apparent that this vanishing coefficient phenomenon occurred only for primes $p \equiv 1 \pmod{4}$. Further investigation revealed that it occurred for pairs of exponents $(j,k)$ of the form  $(j,k)=(bm,bn)$, where $p=m^2+n^2$ and $b\in \N$. These investigations suggested that in this situation, $a_{pt+cb^2+db}^{(bm,bn)}=0$ for all integers $t$ and some pair of integers $c$ and $d$ and also allowed us to guess formulae for $c$ and $d$. In particular, for $p\equiv 1 \pmod{12}$ one pair was $(c,d)=(0,w)$, where $w=\bar{2}(m+n)\pmod{p}$.

One method of proof that occurred to us was the ``involution'' described above at  \eqref{bfmap}--\eqref{bfAbfbprops4}. For example, for a given prime $p\equiv 1 \pmod{12}$ one finds $m$ and $n$ such that $p=m^2+n^2$ and $3\mid m$ and selects one of the $\tau$ found experimentally, such as $\tau=w:=\bar{2}(m+n)$. Then using a computer algebra system such as \emph{Mathematica} one follows through the steps described in the introduction as far as computing $g(\bfx)$ as described at \eqref{grsuveq}.

Next, one sets up a matrix $\bfA$ with entries $a_{i,j}$ to be determined, and likewise a vector $\bfb$ with entries $b_k$ to be determined. One computes $\bfx' := \bfA\bfx+\bfb$ as at \eqref{bfmap} and then the difference $g(\bfx') - g(\bfx)$, which we wish to be identically zero. Regarded as a polynomial in $r$, $s$, $u$ and $v$, setting the coefficients of this polynomial equal to zero gives a system of fifteen quadratic equations in the $a_{i,j}$  and $b_k$ for which we seek integral solutions. 

Each of these equations may have, in some cases, hundreds or even thousands of solutions, so attempting to solve all fifteen equations simultaneously with a single \textbf{Solve[\,]} command is not practical, time-wise. Our initial method was to solve the system one equation at a time (with some preliminary testing to determine which equations are most easily solvable), then once a quadratic equation has been solved, each solution is fed back into the system of equations, which has the effect of making some of the quadratic equations linear, and hence easily solvable for one variable in terms of the other variables. One then feeds these solutions into  the system of equations, moves on to solve another quadratic equation and repeats the process with the resulting linear equations.

With careful management the method will eventually produce two solutions in a reasonable amount of time, one being the trivial solution ($\bfA=\bfI$, the identity matrix, and $\bfb=\bf0$) and the other being the ``involution'' solution we sought. However, the method was speeded up when one of us made the observation that $g(\bfx)$ may be represented as 
\begin{equation}\label{eq:g(x)MatForm}
        g(\bfx) = \bfx^{\tr}\bfQ\bfx + \bfz^{\tr}\bfx + c,
\end{equation}
for some matrix $\bfQ$, vector $\bfz$ and integer $c$. Then $g(\bfx') - g(\bfx) =0$ implies that 
\[
\bfA^{\tr}\bfQ\bfA = \bfQ,
\]
or that $\bfA$ is in the automorphism group of $\bfQ$. There are commands, such as \textbf{qfauto(\,)} in PARI/GP, that will compute the generators of the automorphism group. These can be imported back into \emph{Mathematica} and the full automorphism group generated. One then cycles through the full automorphism group, letting $\bfA$ be each of these matrices in turn, giving numerical values for the $a_{i,j}$ in the fifteen equations mentioned above, and then the remaining equations solved for the $b_k$. 

Regardless of the method used, one then generates sufficient numerical data to identify the various sub-cases in which the matrices $\bfA$ and vectors $\bfb$ have different forms, and sufficient data to identify these forms. What follows is a somewhat idealized description of how experimentation led to the formulae for $\bfA$ and $\bfb$ in one case. We write ``idealized'' as this outline is a result of hindsight, and in reality the progress from experimentation to formulae was less straightforward and involved more experimentation and guesswork. One reason for this, for example, was that for $p=13$ and $\tau =9$ the formula for $v$ at \eqref{vsubeq} gives 
\[
v\to 42 r+2 s-21 u+13 v+63,
\]
whereas our initial experiments reduced coefficients modulo 13 and used instead 
\[
v\to 3 r+2 s+5 u+13 v+11,
\]
which led to different solutions for $\bfA$ and $\bfb$ than those in Table \ref{tab1}, so that the pattern was not nearly so obvious. 

{\allowdisplaybreaks
\begin{center}
\begin{longtable}{C|C|C|C|C|C|C}\caption{The computation of the matrix $\bfA$ and the vector $\bfb$ for various primes $p\equiv 1 \pmod{12}$ and $\tau=w = \bar{2}(m+n)$.}\label{tab1}\\
p&m&n&w&\bfA&\bfb& (f(\bfx')-f(\bfx))\pmod{2}\\
\hline
&&&&&& \\
 13 & 3 & 2 & 9 & \left(
\begin{array}{cccc}
 -13 & 0 & 6 & -4 \\
 20 & 1 & -10 & 6 \\
 12 & 0 & -7 & 4 \\
 60 & 0 & -30 & 19 \\
\end{array}
\right) & \left(
\begin{array}{c}
-19\\29\\20\\87
\end{array}
\right)  & 1 \\
&&&&&& \\
37 & 6 & 1 & 22 & \left(
\begin{array}{cccc}
 -1 & 0 & 0 & -4 \\
 8 & 1 & -4 & 24 \\
 0 & 0 & -1 & 4 \\
 0 & 0 & 0 & 1 \\
\end{array}
\right) &\left(
\begin{array}{c}
-2\\14\\3\\0
\end{array}
\right)  & 1 \\
&&&&&& \\
61 & 6 & 5 & 36 & \left(
\begin{array}{cccc}
 -193 & 0 & 96 & -20 \\
 232 & 1 & -116 & 24 \\
 192 & 0 & -97 & 20 \\
 2784 & 0 & -1392 & 289 \\
\end{array}
\right) & \left(
\begin{array}{c}
-578\\694\\579\\8328
\end{array}
\right)  & 1 \\
&&&&&& \\
73 & 3 & 8 & 42 & \left(
\begin{array}{cccc}
 -85 & 0 & 42 & -16 \\
 32 & 1 & -16 & 6 \\
 84 & 0 & -43 & 16 \\
 672 & 0 & -336 & 127 \\
\end{array}
\right) & \left(
\begin{array}{c}
-589\\221\\590\\4641
\end{array}
\right)  & 1 \\
&&&&&& \\
97 & 9 & 4 & 55 & \left(
\begin{array}{cccc}
 -325 & 0 & 162 & -24 \\
 732 & 1 & -366 & 54 \\
 324 & 0 & -163 & 24 \\
 6588 & 0 & -3294 & 487 \\
\end{array}
\right) & \left(
\begin{array}{c}
-993\\2235\\994\\20115
\end{array}
\right)  & 1 \\
&&&&&& \\
 109 & 3 & 10 & 61 & \left(
\begin{array}{cccc}
 -13 & 0 & 6 & -20 \\
 4 & 1 & -2 & 6 \\
 12 & 0 & -7 & 20 \\
 12 & 0 & -6 & 19 \\
\end{array}
\right) & \left(
\begin{array}{c}
-123\\37\\124\\111
\end{array}
\right)  & 1 \\
\end{longtable}
\end{center}
}

\emph{Mathematica} was used to generate the matrices $\bfA$ and the vectors $\bfb$ in Table \ref{tab1}, for various primes $p\equiv 1 \pmod{12}$ and $\tau=w = \bar{2}(m+n)$. As it turns out, these are exactly what one gets from the formulae for $\bfA$ and  $\bfb$ in \textbf{Case 1} of the proof of Theorem \ref{mainthm}. Of course one does not have those formulae initially and one next proceeds to try to determine them from consideration of commonalities in the matrices and vectors in Table \ref{tab1}. Most obvious of course is column 2 of the matrices. One also observes other common features such as $a_{1,4}=-a_{3,4}=-2mn/3$, etc., etc. Proceeding in this manner and attempting to solve the general system of equations with these clues, one eventually determines the forms for $\bfA$ and  $\bfb$ stated in \textbf{Case 1} of the proof of Theorem \ref{mainthm}. The forms for $\bfA$ and  $\bfb$ stated in the other cases of the proof of Theorem \ref{mainthm} were derived similarly.

\section{Proof of Theorem \ref{mainthm} }
In this section we give our proof of Theorem \ref{mainthm}. For convenience, we restate this theorem.

\begin{manualtheorem}{\textbf{\ref{mainthm}}}
    Let $Q(z,q)$ be as at \eqref{qtpzqdef}, let $p \equiv 1 \pmod 4$ be a prime, let $m>0$, $n>0$ be integers such that $p=m^2+n^2$. Let $b>0$ be an integer, and let the sequence $a_t$ be defined by 
        \(
            Q(q^{bm},q^p)Q(q^{bn},q^p) = \sum_{t=0}^{\infty} a_t q^t.
        \)
    Let
        \[
            w \equiv \bar{2}(m+n) \pmod{p}.
        \]
\begin{enumerate}[label={\normalfont (\arabic*)}]
\item%
    If $p \equiv 1 \pmod{12}$, one has
        \[
            a_{pt+bw} = a_{pt+(w-3b)b} = 0 \qquad (\text{for all integers $t$}).
        \]
\item%
    If $p \equiv 5 \pmod{12}$, one has
        \[
            a_{pt+(1-3b\bar{m})bw} = a_{pt+(1-3b\bar{n})bw} = 0 \qquad (\text{for all integers $t$}).
        \]
\end{enumerate}
\end{manualtheorem}

\begin{proof}
 We continue with the notation introduced in the introduction, and following the discussion in the introduction, all that needs to be done is to state the form of the matrices $\bfA$ and $\bfb$ in each of the four cases, verify that the entries are integers, and that the properties listed at \eqref{bfAbfbprops1}--\eqref{bfAbfbprops4} are satisfied. The latter is most easily achieved using a computer algebra system such as \emph{Mathematica}, and a notebook that does this is available at the web site of one of the authors \cite{McL}.

\noindent\textbf{Case 1:} Let $p \equiv 1 \pmod{12}$ and let $\tau = w = \bar{2}(m+n)$. As $p=m^2+n^2$ by assumption, it is evident that one, and only one, of $m$ and $n$ must be equivalent to $0$ (mod $3$). Thus, suppose that $3\,|\, m$.
    
In this case we have
    \begin{align*}
    \bfA & = 
    {\displaystyle\left(
\begin{array}{cccc}
 -\frac{4 m^2 \left(n \bar{n}-1\right)}{3 p}-1 & 0 & \frac{2 m^2 \left(n \bar{n}-1\right)}{3 p} & -\frac{2 m n}{3}  \\
 \frac{4 m \bar{n}}{3}-\frac{4 m n \left(n \bar{n}-1\right)}{3 p} & 1 & \frac{2 m n \left(n \bar{n}-1\right)}{3 p}-\frac{2 m \bar{n}}{3} & \frac{2 m^2}{3} \\
 \frac{4 m^2 \left(n \bar{n}-1\right)}{3 p} & 0 & -\frac{2 m^2 \left(n \bar{n}-1\right)}{3 p}-1 & \frac{2 m n}{3} \\
 \frac{4 m \left(n \bar{n}-1\right) \left(p \bar{n}-n \left(n \bar{n}-1\right)\right)}{p^2} & 0 & -\frac{2 m \left(n \bar{n}-1\right) \left(p \bar{n}-n \left(n \bar{n}-1\right)\right)}{p^2} & \frac{2 m^2 \left(n \bar{n}-1\right)}{p}+1 \\
\end{array}
\right)},
    \end{align*}
and
    \begin{align*}
    \bfb &= 
       \left(
\begin{array}{c}
 \frac{m \left(-2 n w \bar{n}+m+n\right)}{3 p} \\
 \frac{m \left(2 m w \bar{n}-m+n\right)}{3 p} \\
 1-\frac{m \left(-2 n w \bar{n}+m+n\right)}{3 p} \\
 \frac{m \left(n \bar{n}-1\right) \left(2 m w \bar{n}-m+n\right)}{p^2} \\
\end{array}
\right).
    \end{align*}
That the various matrix- and vector entries are integers follows from the facts that $3\mid m$, $w \equiv \bar{2}(m+n)\pmod{p}$ and $n\bar{n}\equiv 1 \pmod{p}$, together with the congruence $2 m w \bar{n} \equiv m-n \pmod{p}$, which easily follows from $w \equiv \bar{2}(m+n)\pmod{p}$.
    
\vspace{\baselineskip}

\noindent\textbf{Case 2:} Here also $p \equiv 1 \pmod{12}$, and this time  $\tau =w-3b$; we again take $3\mid m$. 
    \begin{align*}
    \bfA &= \left(
\begin{array}{cccc}
 \frac{4 m^2 \left(n \bar{n}-1\right)}{3 p}+1 & 0 & -\frac{2 m^2 \left(n \bar{n}-1\right)}{3 p} & \frac{2 m n}{3} \\
 \frac{4 m n \left(n \bar{n}-1\right)}{3 p}-\frac{4 m \bar{n}}{3} & -1 & \frac{2 m \bar{n}}{3}-\frac{2 m n \left(n \bar{n}-1\right)}{3 p} & -\frac{2 m^2}{3} \\
 -\frac{4 m^2 \left(n \bar{n}-1\right)}{3 p} & 0 & \frac{2 m^2 \left(n \bar{n}-1\right)}{3 p}+1 & -\frac{2 m n}{3} \\
 -\frac{4 m \left(n \bar{n}-1\right) \left(p \bar{n}-n \left(n \bar{n}-1\right)\right)}{p^2} & 0 & \frac{2 m \left(n \bar{n}-1\right) \left(p \bar{n}-n \left(n \bar{n}-1\right)\right)}{p^2} & -\frac{2 m^2 \left(n \bar{n}-1\right)}{p}-1 \\
\end{array}
\right),
    \end{align*}
and
    \begin{align*}
    & \bfb =         \left(
\begin{array}{c}
 -\frac{m \left(-2 n \bar{n} (w-3 b)-6 b+m+n\right)}{3 p} \\
 -\frac{2 m^2 \bar{n} (w-3 b)-6 b n-m^2+m n}{3 p} \\
 \frac{m \left(-2 n \bar{n} (w-3 b)-6 b+m+n\right)}{3 p} \\
 -\frac{\left(m^2 \bar{n}+n\right) \left(2 n \bar{n} (w-3 b)+6 b-m-n\right)}{p^2} \\
\end{array}
\right).
    \end{align*}
 That the various matrix- and vector entries are integers follow from similar considerations to those in \textbf{Case 1}, this time also using $m^2\equiv -n^2 \pmod{p}$.   
\begin{remark}
    Note that the matrix $\bfA$ in \textbf{Case 2} is the negative of the matrix $\bfA$ in \textbf{Case 1}.
\end{remark}

Before coming to primes $p \equiv 5 \pmod{12}$ we note that the division into cases is different. Instead of the two cases being derived from the vanishing in the two different arithmetic progressions, as in the case of primes $p \equiv 1 \pmod{12}$, we note that if  $p \equiv 5 \pmod{12}$, $p=m^2+n^2$, then either $3\mid (m-n)$ (\textbf{Case 3}) or 
$3\mid (m+n)$ (\textbf{Case 4}). We also note that in either of these cases, it is sufficient to demonstrate the matrix and vector that shows vanishing in the arithmetic progression $a_{pt+(1-3b\bar{m})bw}$, since we can freely switch $m$ and $n$ to get vanishing in the arithmetic progression $a_{pt+(1-3b\bar{n})bw}$ (see the statement of Theorem \ref{mainthm}).

\vspace{\baselineskip}

\noindent\textbf{Case 3:} Let $p \equiv 5 \pmod{12}$ and $3\mid (m-n)$, and we consider $a_{pt+(1-3b\bar{m})bw}$, so $\tau = (1-3b\bar{m})w$. This time,
    \begin{align}\label{mainthmprf2case3}
        &\bfA=\left(
\begin{array}{cccc}
 \frac{2 m \bar{n} (m-n) (m+n)+4 m n+p}{3 p} & 0 & -\frac{(m-n) \left(m^2 \bar{n}+m n \bar{n}-m+n\right)}{3 p} & \frac{1}{3} (m-n) (m+n) \\
 -\frac{2 (m-n) \left(m^2 \bar{n}-m n \bar{n}+m+n\right)}{3 p} & -1 & \frac{(m-n) \left(m^2 \bar{n}-m n \bar{n}+m+n\right)}{3 p} & -\frac{1}{3} (m-n)^2 \\
 -\frac{2 (m-n) \left(m^2 \bar{n}+m n \bar{n}-m+n\right)}{3 p} & 0 & \frac{m \bar{n} (m-n) (m+n)+2 m n+2 p}{3 p} & -\frac{1}{3} (m-n) (m+n) \\
 a_{41}
  & 0 & 
 a_{43}
  & -\frac{m \left(m^2 \bar{n}-n^2 \bar{n}+2 n\right)}{p} \\
\end{array}
\right) \\
    &\bfb = \left(
\begin{array}{c}
 -\frac{(m-n) \left(w \bar{n} (m+n) \left(3 b \bar{m}-1\right)-3 b+m\right)}{3 p} \\
 \frac{w \bar{n} (m-n)^2 \left(3 b \bar{m}-1\right)+3 b (m+n)+n (n-m)}{3 p} \\
 \frac{(m-n) \left(w \bar{n} (m+n) \left(3 b \bar{m}-1\right)-3 b+m\right)}{3 p} \\
 \frac{\left(m \bar{n} (m-n)+m+n\right) \left(w \bar{n} (m+n) \left(3 b \bar{m}-1\right)-3 b+m\right)}{p^2} \\
\end{array}
\right),
    \end{align}
where
    \begin{align*}
        a_{41} & =-\frac{2 \left(m^2 \bar{n}-m n \bar{n}+m+n\right) \left(m^2 \bar{n}+m n \bar{n}-m+n\right)}{p^2},  \\
        a_{43} & =  \frac{\left(m^2 \bar{n}-m n \bar{n}+m+n\right) \left(m^2 \bar{n}+m n \bar{n}-m+n\right)}{p^2}.
    \end{align*}
The integrality of the terms follows by arguments similar to those stated previously, noting that in this case $3\mid 4 m n+p$.

\begin{remark}
    The matrix \textbf{A} above, and also in \textbf{Case 4} below, is displayed as shown because of the width of the displayed entries.

\end{remark}

\vspace{\baselineskip}

\noindent\textbf{Case 4:} Lastly, we consider $p \equiv 5 \pmod{12}$ and $3\mid (m+n)$, we again consider $a_{pt+(1-3b\bar{m})bw}$, so once more $\tau = (1-3b\bar{m})w$. In this case 
    \begin{align}\label{mainthmprf2case4}
    &\bfA = \left(
\begin{array}{cccc}
 \frac{2 m \bar{n} (m-n) (m+n)+4 m n-p}{3 p} & 0 & -\frac{(m+n) \left(m^2 \bar{n}-m n \bar{n}+m+n\right)}{3 p} & \frac{1}{3} (m-n) (m+n) \\
 \frac{2 (m+n) \left(m^2 \bar{n}+m n \bar{n}-m+n\right)}{3 p} & 1 & -\frac{(m+n) \left(m^2 \bar{n}+m n \bar{n}-m+n\right)}{3 p} & \frac{1}{3} (m+n)^2 \\
 -\frac{2 (m+n) \left(m^2 \bar{n}-m n \bar{n}+m+n\right)}{3 p} & 0 & \frac{m \bar{n} (m-n) (m+n)+2 m n-2 p}{3 p} & -\frac{1}{3} (m-n) (m+n) \\
a_{4,1}
 & 0 & 
 a_{4,3}
 & -\frac{m \left(m^2 \bar{n}-n^2 \bar{n}+2 n\right)}{p} \\
\end{array}
\right)\notag\\
    &\bfb= 
       \left(
\begin{array}{c}
 \frac{w \bar{n} \left(n^2-m^2\right) \left(3 b \bar{m}-1\right)+3 b (m-n)+n (m+n)}{3 p} \\
 -\frac{(m+n) \left(w \bar{n} (m+n) \left(3 b \bar{m}-1\right)-3 b+m\right)}{3 p} \\
 \frac{w \bar{n} (m-n) (m+n) \left(3 b \bar{m}-1\right)+3 b (n-m)-n (m+n)+3 p}{3 p} \\
 \frac{\left(m \bar{n} (m-n)+m+n\right) \left(w \bar{n} (m+n) \left(3 b \bar{m}-1\right)-3 b+m\right)}{p^2} \\
\end{array}
\right),
    \end{align}
    \normalsize
where
    \begin{align*}
        a_{4,1} & =
         -\frac{2 \left(m^2 \bar{n}-m n \bar{n}+m+n\right) \left(m^2 \bar{n}+m n \bar{n}-m+n\right)}{p^2} ,  \\
        a_{4,3} & =  \frac{\left(m^2 \bar{n}-m n \bar{n}+m+n\right) \left(m^2 \bar{n}+m n \bar{n}-m+n\right)}{p^2}.
    \end{align*}
 The integrality of the terms follows  similarly, noting that in this case $3\mid (4 m n-p)$.
\end{proof}

\section{Some Examples}
We next illustrate the results in Theorem \ref{mainthm} with some examples.

\begin{example}\label{exthem1}
(1) Let $p=37=6^2+1^2$ (so take $m=6$ and $n=1$), let $b>0$, and define the sequence $a_t$ by 
    \[
        Q(q^{6b},q^{37})Q(q^{b},q^{37})=\sum_{t=0}^{\infty}a_tq^t.
    \]
Since $37 \equiv 1 \pmod{12}$, we have 
    \[
        w \equiv \bar{2}(m+n) = \bar{2}(6+1) \equiv 22 \pmod{37}.
    \]
Then for all integers $t$,
    \[
        a_{37t+22b} = a_{37t+22b-3b^{2}} = 0.
    \]
(2) Let  $p=29=5^2+2^2$ (so take $m=5$ and $n=2$), let $b > 0$, and define the sequence $a_t$ by 
    \[
        Q(q^{5b},q^{29})Q(q^{2b},q^{29}) = \sum_{t=0}^{\infty}a_tq^t.
    \]
Since $29 \equiv 5 \pmod{12}$, 
{\allowdisplaybreaks
    \begin{align*}
        w \equiv \bar{2}(m+n) &\equiv 18 \pmod{29}, \\
         -3\bar{m}w &\equiv 24 \pmod{29}, \\
        -3\bar{n}w &\equiv 2 \pmod{29}.
    \end{align*}
}
Then for all integers $t$,
    \[
        a_{29t+24b^2+18b} = a_{29t+2b^2+18b} = 0.
    \]
\end{example}

\begin{remark} For any prime $p \geq 5$ with $p = m^2+n^2$, one has by \eqref{qpideq} that both $Q(q^{bm},q^p)$ and $Q(q^{bn},q^p)$ are \emph{superlacunary}, which Ono and Robins \cite{OR95} defined to mean that if either is expanded as a series in $q$, it has the form
    \[
        \sum_{n=-\infty}^{\infty} d(an^2+bn+c) q^{an^2+bn+c}
    \]
where $a,b,c$ are rational and $a>0$ (Ono and Robins take $a,b,c$  to be integers, since replacing $q$ with $q^N$ for some integer $N$ would clear all denominators). The authors point out in \cite{OR95} that the product of two superlacunary series is lacunary and hence the product $ Q(q^{bm},q^p)Q(q^{bn},q^p)$ has a series expansion that is lacunary, and almost all the coefficients $a_t$ in Theorem \ref{mainthm} are zero.

\end{remark}

However, multiplying $ Q(q^{bm},q^p)Q(q^{bn},q^p)$  by any function whose series expansion is a series in $q^p$ will not affect the vanishing of coefficients in arithmetic progressions modulo $p$, so we state the following corollary, which may have more interesting combinatorial implications on special partition functions.

\begin{corollary}\label{maincor}
    Let $Q(z,q)$ be as at \eqref{qtpzqdef}, let $p \equiv 1 \pmod 4$ be a prime, and $m>0$, $n>0$ be integers such that $p=m^2+n^2$. Let $b>0$ be an integer, let
        \[
            w\equiv \bar{2}(m+n)\pmod{p},
        \]
    and let the sequence $c_t$ be defined by 
        \(\label{maincoreq1}
            \frac{Q(q^{bm},q^p)Q(q^{bn},q^p)}{(q^p;q^p)_{\infty}^2}=\sum_{t=0}^{\infty} c_t q^t.
        \)
\begin{enumerate}[label={\normalfont (\arabic*)}]
\item%
    If $p \equiv 1 \pmod{12}$, then for all integers $t$ one has
        \(\label{maincoreq2}
            c_{pt+bw} = c_{pt+b(w-3b)} = 0.
        \)
\item%
    If $p \equiv 5 \pmod{12}$, then for all integers $t$ one has
        \(\label{maincoreq3}
            c_{pt+bw(1-3b\bar{m})} = c_{pt+bw(1-3b\bar{n})} = 0.
        \)
\end{enumerate}
\end{corollary}

As an application of Corollary \ref{maincor}, we modify the products in Example \ref{exthem1} and consider just the case $b=1$ (although the results translate directly for an arbitrary positive integer $b$).

\begin{example}\label{excor1}
(1) Let  the sequence $a_t$ be defined by 
    \[
    \frac{Q(q,q^{37})Q(q^{6},q^{37})}{(q^{37};q^{37})_{\infty}^2}  =  (q,q^6,q^{31},q^{26};q^{37})_{\infty} (q^{25},q^{35},q^{39},q^{49};q^{74})_{\infty} = \sum_{t=0}^{\infty} a_t q^t.
    \]
Then for all integers $t$, one has
    \[
        a_{37t+22} = a_{37t+19} = 0.
    \]
(2) Let  the sequence $a_t$ be defined by 
    \[
      \frac{Q(q^{2},q^{29})Q(q^{5},q^{29})}{(q^{29};q^{29})_{\infty}^2}=  (q^2,q^5,q^{24},q^{27};q^{29})_{\infty} (q^{19},q^{25},q^{33},q^{39};q^{58})_{\infty} = \sum_{t=0}^{\infty}a_t q^t.
    \]
Then for all integers $t$, one has
    \[
        a_{29t+42} = a_{29t+20} = 0.
    \]
\end{example}

\begin{remark}
    Unlike the coefficients in the series expansions in Example \ref{exthem1}, most coefficients in the series expansions in Example \ref{excor1} are non-zero. Indeed it follows from a general result proven in \cite{DHMcLY26} that apart from the two arithmetic progressions where coefficients vanish, the coefficients in the other arithmetic progressions are, apart from a few possible zero coefficients near the beginning, eventually either all positive or all negative, tending to either $+\infty$ or $-\infty$.
\end{remark}

\section{Concluding Remarks}

Computer investigations suggest that vanishing coefficient results similar to those proven in the present paper for products of two quintuple products also hold for products of three- and four quintuple products.

For example, if the sequence $\{a_n\}$ is defined by
\[
Q(q,q^{31})Q(q^5,q^{31})Q(q^6,q^{31})=\sum_{n=0}^{\infty}a_nq^n,
\]
then experiment suggests that 
\[
a_{31t+6}=a_{31t+16}=a_{31t+18}=0.
\]
 Similarly, if the sequence $\{b_n\}$ is defined by
\[
Q(q,q^{41})Q(q^4,q^{41})Q(q^5,q^{41})Q(q^9,q^{41})=\sum_{n=0}^{\infty}b_nq^n,
\]
then it appears that 
\[
b_{41t+21}=b_{41t+25}=b_{41t+26}=b_{41t+30}=0.
\]

We will examine these phenomena in a subsequent paper.


{\allowdisplaybreaks

}


\begin{thebibliography}{99}



\bibitem{AG94}
Alladi, K.; Gordon B. \emph{Vanishing coefficients in the expansion of products of Rogers-Ramanujan type.} Proc. Rademacher Centenary Conference, (G. E. Andrews and D. Bressoud, Eds.), Contemp. Math. \textbf{166}, (1994), 129--139.


\bibitem{AB79}
Andrews, G. E.; Bressoud, D. M.
\emph{Vanishing coefficients in infinite product expansions.}
J. Austral. Math. Soc. Ser. A \textbf{27} (1979), no. 2, 199–-202.

\bibitem{BK19}
Baruah, N.D.; Kaur, M. 
\emph{Some results on vanishing coefficients in infinite product expansions}.
Ramanujan J. \textbf{53} (2020), no. 3, 551--568.





\bibitem{CT20}
    Chern, S.; Tang, D.
\emph{Vanishing coefficients in quotients of theta functions of modulus five.}
Bull. Aust. Math. Soc. \textbf{102} (2020), no. 3, 387--398.

\bibitem{CT24}
    Chern, S.; Tang, D.
\emph{General coefficient-vanishing results associated with theta series.}
Adv. in Appl. Math. \emph{159} (2024), Paper No. 102742, 66 pp.

\bibitem{C06}
Cooper, S.
\emph{The quintuple product identity.}
Int. J. Number Theory \textbf{2} (2006), no. 1, 115--161.

\bibitem{D24}
Daniels, T.
\emph{Vanishing coefficients in two q-series related to Legendre-signed partitions.}
Res. Number Theory \textbf{10} (2024), no. 4, Paper No. 81, 12 pp.



\bibitem{DHMcLY26}
Daniels, T.; Huber, T.; McLaughlin, J.; Ye, D. 
\emph{The $p$-Dissection of a Product of Quintuple Products} - submitted


\bibitem{Hirschhorn2017}
Hirschhorn, M.~D.  \emph{The power of $q$: A personal journey}, Springer, 2017.




\bibitem{KV22}    
Kaur, M.; Vanda.
\emph{Results on vanishing coefficients in infinite q-series expansions for certain arithmetic progressions mod7.}
Ramanujan J. \textbf{58} (2022), no. 1, 269–289.


\bibitem{L24}
Liu, J.-C.
\emph{On the vanishing coefficients of odd powers of Ramanujan's theta functions.}
Ramanujan J. \textbf{65} (2024), no. 1, 45--52.



\bibitem{McL}
McLaughlin, J. Mathematica notebook containing the calculations that verify the different parts of Theorem \ref{mainthm} available here:\\ \emph{\href{https://www.wcupa.edu/sciences-mathematics/mathematics/jMcLaughlin/mathnbs.aspx}{https://www.wcupa.edu/sciences-mathematics/mathematics/jMcLaughlin/mathnbs.aspx}
}.

\bibitem{McL15}
McLaughlin, J.
\emph{Further results on vanishing coefficients in infinite product expansions}, 
{J. Aust. Math. Soc.} \textbf{98} (2015), no. 1, 69--77.

  \bibitem{M19a}
McLaughlin,  J. 
 \emph{New infinite q-product expansions with vanishing coefficients.} Ramanujan J \textbf{55}, 733--760 (2021). 

 \bibitem{McLZ22}
McLaughlin,  J. and Zimmer, P.
\emph{Further results on Vanishing Coefficients in  infinite products of the form  
       $(q^b,q^{p-b};q^p)_{\infty}^3(q^{jb},q^{2p-jb};q^{2p})_{\infty}$}
Int. J. Number Theory \textbf{18} (2022), no. 8, 1863--1885.

\bibitem{OR95}
Ono, K.; Robins, S.
Superlacunary cusp forms.
\emph{ Proc. Amer. Math. Soc.} \textbf{123} (1995), no. 4, 1021--1029.



\bibitem{RS78}
Richmond, B.; Szekeres, G.
\emph{The Taylor coefficients of certain infinite products. }
Acta Sci. Math. (Szeged) \textbf{40} (1978), no. 3--4, 347–-369.

\bibitem{ST23}
Somashekara, D. D.; Thulasi, M. B.
\emph{Results on vanishing coefficients in certain infinite q-series expansions.}
Ramanujan J. \textbf{60} (2023), no. 2, 355--369.


 \bibitem{T19}
Tang,  D. 
  \emph{Vanishing coefficients in some q-series expansions.}
   Int. J. Number Theory \textbf{15} (2019), no. 4, 763--773.

    \bibitem{T19b}
Tang,  D. 
\emph{Vanishing coefficients on four quotients of infinite product expansions.}
Bull. Aust. Math. Soc. \textbf{100} (2019), no. 2, 216–224.

 \bibitem{T23}
Tang,  D. 
\emph{Vanishing coefficients in three families of products of theta functions.}
Rev. R. Acad. Cienc. Exactas Fís. Nat. Ser. A Mat. RACSAM \textbf{117} (2023), no. 1, Paper No. 36, 11 pp.

 \bibitem{T23b}
Tang,  D. 
\emph{Vanishing coefficients in three families of products of theta functions. II.}
Results Math. \textbf{78} (2023), no. 4, Paper No. 124, 28 pp.

 \bibitem{T24}
Tang,  D. 
\emph{Vanishing coefficients in powers of theta functions with odd moduli.}
Rocky Mountain J. Math. \textbf{54} (2024), no. 1, 261--267.


\bibitem{VK22}
Vanda; Kaur, M.
\emph{Vanishing coefficients of $q^{5n+r}$ and $q^{11n+r}$ in certain infinite q-product expansions.}
Ann. Comb. \textbf{26} (2022), no. 3, 533--557.


\bibitem{XiaZhao2022}
Xia, Ernest X. W.; Zhao, Alice X. H. 
\emph{Generalizations of Hirschhorn's Results on Two Remarkable $q$-Series Expansions},
\newblock { Experimental Mathematics}, 31(3):878--882, 2022.

\end{thebibliography}
\end{document}